\documentclass[a4paper,11pt,reqno,noindent]{amsart}
\usepackage[centertags]{amsmath}
\usepackage{mathtools}
\usepackage{amsfonts,amssymb,amsthm,dsfont,cases,amscd,esint,enumerate}
\usepackage[T1]{fontenc}
\usepackage[english]{babel}
\usepackage{psfrag}
\usepackage{newlfont,accents}
\usepackage{color}
\usepackage[body={15cm,21.5cm},centering]{geometry} 
\usepackage{fancyhdr}

\usepackage{enumerate}

\mathchardef\emptyset="001F
\theoremstyle{plain}
\newtheorem{theorem}{Theorem}
\newtheorem{lemma}{Lemma}

\theoremstyle{definition}

\theoremstyle{remark}
\newtheorem{remark}{Remark}

\newcommand{\N}{\mathbb N}

\newcommand{\calQ}{\mathcal{Q}}
\newcommand{\calE}{\mathcal{E}}
\newcommand{\calZ}{\mathcal{Z}}

\newcommand{\calC}{\mathcal{C}}

\newcommand{\calT}{\mathcal T}

\newcommand{\R}{\mathbb R}
\newcommand{\Z}{\mathbb Z}

\newcommand{\loc}{{\operatorname{loc}}}
\newcommand{\uloc}{{\operatorname{uloc}}}

\newcommand{\supp}{\operatorname{supp}}
\newcommand{\diam}{\operatorname{diam}}

\newcommand{\step}[1]{\noindent \textit{Step} #1.}

\newcommand{\expec}[1]{\mathbb{E}\left[ #1 \right]}
\newcommand{\expecm}[1]{\mathbb{E}\big[ #1 \big]}

\newcommand{\cov}[2]{\operatorname{Cov}\left[{#1};{#2}\right]}

\newcommand{\RR}{{\mathbb{R}}}

\newcommand{\ZZ}{{\mathbb{Z}}}

\newcommand{\bP}{\mathbb{P}}

\usepackage{color}
\usepackage[colorlinks,citecolor=black,urlcolor=black]{hyperref}

\usepackage[normalem]{ulem}

\begin{document}
\title{The landscape function on $\R^d$}

\begin{abstract}
Consider the Schr\"odinger operator $-\triangle+\lambda V$ with non-negative iid random potential $V$ of strength $\lambda>0$.
We prove existence and uniqueness of the associated landscape function on the whole space, and show that its correlations decay exponentially.
As a main ingredient we establish the (annealed and quenched) exponential decay of the Green function of $-\triangle+\lambda V$ using Agmon's positivity method, rank-one perturbation in dimensions $d\ge 3$, and first-passage percolation in dimensions $d=1,2$.
\end{abstract}

\selectlanguage{english}
\author{G. David, A. Gloria, and S. Mayboroda.}
\newcommand{\Addresses}{{ 
  \bigskip
  \vskip 0.08in \noindent --------------------------------------
\vskip 0.10in

  \footnotesize

  G.~David, \textsc{Universit\'e Paris-Saclay, Laboratoire de Math\'{e}matiques d'Orsay, 91405, France
}\par\nopagebreak
  \textit{E-mail address}: \texttt{Guy.David@universite-paris-saclay.fr}

 \medskip
 
A.~Gloria, \textsc{LJLL, Sorbonne Universit\'e, 4 place Jussieu, 75005 Paris France, Institut Universitaire de France (IUF), and D\'epartement de Math\'ematique, ULB, Bruxelles, Belgium}\par\nopagebreak
  \textit{E-mail address}: \texttt{antoine.gloria@sorbonne-universite.fr}

\medskip

S.~Mayboroda, \textsc{School of Mathematics, University of Minnesota, 206 
Church St SE, Minneapolis, MN 55455 USA}\par\nopagebreak
  \textit{E-mail address}: \texttt{svitlana@math.umn.edu}
}}
\date{}

\maketitle

\tableofcontents

\section{Motivation and the main result}
\subsection{The landscape}
In the recent years the landscape function introduced in \cite{MR2990982} proved to be an indispensable tool both for the theoretical study of the localization of waves in disordered media and in many physical and engineering applications where the localization phenomena play a decisive role. Roughly speaking, the emerging principle is that in order to understand the behavior of waves governed by the Schr\"odinger Hamiltonian $H=-\Delta+V$ with some disordered potential $V$, one should first compute the landscape function, a solution to $Hu=1$, and then use its reciprocal, $1/u$, as an {\it{effective potential}}. The latter is much more illustrative than the original $V$, accurately accounting for both quantum and classical effects. For instance, the corresponding Agmon distance predicts exponential decay in disordered potentials \cite{MR3995095} and the corresponding counting function gives a non-asymptotic estimate on the integrated density of states across the entire spectrum \cite{MR4298594}.

However, virtually all of the existing work addresses the passage from the properties of the landscape to those of the spectrum or the eigenfunctions of the corresponding Hamiltonian, and this is only a half of the puzzle. One naturally arrives at the question about the properties of the landscape as a probabilistic object, e.g., in standard disordered potentials inducing Anderson localization. This is exactly the goal of this paper. 

To this end, let us 
 introduce some 
classical Anderson-type potentials. Let $\varphi\in C_0^\infty(B_{1/10}(0))$
be a nontrivial bump function supported in the
ball centered at 0 of radius $1/10$, with $0\leq \varphi\leq 1$,  and set 
\begin{equation}\label{e.def-V}
V=V_{\omega} (x)=\sum_{j\in \ZZ^d} \omega_j \varphi(x-j) \quad \hbox{for } x\in \RR^d,
\end{equation}
where the $\omega_j$ are i.i.d. variables taking values in $[0,1]$, 
whose probability distribution
$$F(\delta)=\bP[\omega \leq \delta], \quad 0\leq \delta\leq 1,$$
is not trivial, i.e., not concentrated at one point, and such that $0$ is 
the infimum of its support. Note that we do not require any smoothness of the probability distribution, and in particular, we treat atomic (e.g., Bernoulli) potentials under the same hat. The Hamiltonian in the center of our attention will be $H=-\Delta +\lambda V$, and we will carefully track the dependence on the strength of the disorder, the real parameter $\lambda>0$.

Consider the landscape function on $\R^d$, that is, the solution to the equation
\begin{equation}\label{e.landscape}
-\triangle u + \lambda Vu =1.
\end{equation}
As  mentioned above, it was originally introduced in \cite{MR2990982}, on bounded domains with homogeneous boundary conditions, in the setting where, e.g., its existence is deterministic and standard. However, to relate the landscape function to spectral properties of random operators on the whole space, and to properly describe its probabilistic properties, it is important to define it on the whole $\RR^d$ as well.

Recall that a random field $v$ (that is, a measurable map of both the space variable $x$ 
and the randomness $\omega$), is called stationary iff for all $x\in \R^d$ and $z\in \Z^d$
it satisfies
$$
v(x+z,\omega)=v(x,\tau_z \omega), \text{ where for all }j\in \Z^d, (\tau_z \omega)_j = \omega_{j+z}.
$$
Informally one expects that $u(x)=\int  G(x,y)dy$, where  $G$ would denote the Green function of the Schr\"odinger operator $-\triangle + \lambda V$.
The latter is however not easily defined (especially in dimensions 1 and 2), and uniform deterministic bounds do not ensure that it is integrable.
It is therefore not clear a priori that \eqref{e.landscape} admits a stationary solution in the sense above. Proving the existence of the landscape, along with the sharp bounds on its moments and correlations, is one of  the main objectives of this note.

In what follows, we adopt Hardy's notation: $C\ge 1 \ge c>0$ denote positive constants (that might change from line to line) depending only on $d$, $\varphi$, and $F$ -- and not on $\lambda$. When the constants in question, in addition, depend on $\lambda$ and/or the forthcoming integrability parameter $p$, we make it explicit in notation by writing, e.g., $C_{\lambda, p}.$ We also write $ A \lesssim B $ to signify  $A \le  CB$ for some constant $C>0$ with the same dependence, and add subscripts to specify a dependence on further parameters.

\medskip

Our first result ensures the well-posedness of \eqref{e.landscape}.
\begin{theorem}\label{tLand} There exists a
unique stationary solution of the landscape equation~\eqref{e.landscape} with finite second moment. 
 That is, for almost all $\omega$, $u$ lies in the  Sobolev space $H^1_{loc}(\R^d)$ and is a weak solution
of \eqref{e.landscape}, and $\expecm{ \int_Q u^2} < +\infty$. 
In addition, it is positive and satisfies   
\begin{equation}\label{eq.mom}
\expecm{\sup_Q u^p}^\frac1p \leq C_{\lambda, p},
\end{equation}
for all $p\ge 1$. 
 Here again 
$Q:=[-\frac12,\frac12)^d$.
\end{theorem}
Next we investigate the correlations of the landscape function, and establish the following.
\begin{theorem}\label{tLandCorr}
The stationary solution $u$ of~\eqref{e.landscape} satisfies for all $x,y\in \R^d$ 
\begin{eqnarray*}
|\cov{u(x)}{u(y)}|&\le &C_\lambda \exp(-c \,(\sqrt \lambda\wedge 1) |x-y|),\\
\Bigl|\cov{\frac1{u}(x)}{\frac 1u(y)}\Bigr|&\le& C_\lambda\exp(-c\,  (\sqrt \lambda\wedge 1) |x-y|),\\
|\cov{\nabla \log u(x)}{\nabla \log u(y)}|&\le&C_\lambda \exp(-c\,  (\sqrt \lambda\wedge 1) |x-y|),
\end{eqnarray*}
where $\cov{}{}$ denotes the covariance, and where the multiplicative positive constant $c$ depends on the dimension and the distribution $F$ of the random variable, and $C_\lambda$ depends additionally on~$\lambda$.
\end{theorem}
To establish these results, we first consider an infra-red regularization of the landscape function $u_\eta$, $\eta>0$, such that
\begin{equation}\label{e.landscape-mu}
-\triangle u_\eta + (\lambda V+\eta)u_\eta =1,
\end{equation}
which is well-posed on a deterministic level.
We shall prove Theorem~\ref{tLand} by establishing moment bounds on $u_\eta$
which are uniform with respect to $\eta$ and by passing to the limit $\eta \downarrow 0$ in \eqref{e.landscape-mu}.
Likewise, Theorem~\ref{tLandCorr} is obtained by passing to the limit $\eta \downarrow 0$ in the corresponding estimates for $u_\eta$.
These uniform estimates rely on uniform decay estimates for the infra-red regularization of the Green function, which are the object of Section~\ref{sec:Green} below.
Theorems~\ref{tLand} and~\ref{tLandCorr} are then obtained by a post-processing in 
Section~\ref{sec:landscape}. 

\subsection{Exponential decay of the Green function}\label{sec:Green} The present section is devoted to the exponential decay estimates on the Green function which we hope will be of independent interest. 

For all $\eta>0$, denote by $G_\eta$ the Green function of $-\triangle +(\lambda V+\eta)$, defined
for all $x\in \R^d$ as the
unique decaying solution of
\begin{equation}\label{defGr}
(-\triangle_y +(\lambda V+\eta))G_\eta (x,y)=\delta(x-y).
\end{equation}
The main result of this section is the following exponential decay of $G_\eta$, uniformly with respect to $\eta>0$.
\begin{theorem}\label{tGrAnt} 
There exists $c>0$ such that for all $p\ge 1$, $\eta>0$ and  $x,y\in \R^d$, we have
\begin{equation}\label{e.annealed-mass}
0 \le \expec{\big(\int_Q G_{\eta,\lambda}(x,y+y')dy'\big)^{p}}^\frac1{p} \le  \,C_{\lambda,p} \exp\big(-\tfrac cp  (\sqrt{\lambda}\wedge 1) |x-y|\big) ,
\end{equation}
where the constant $C_{\lambda,p}$, which only depends on $\lambda$, $p$, and $d$,  is made explicit in the proofs.
\end{theorem}
\begin{remark}
Since the estimate~\eqref{e.annealed-mass} is uniform in $\eta>0$, one can pass to the limit as $\eta \downarrow 0$ and uniquely define the Green function $G_\lambda$ of the random operator $-\triangle +\lambda V$ as an almost surely exponentially decaying function (that satisfies the estimate~\eqref{e.annealed-mass}). We refer the reader to Section~\ref{sec:0} for a similar argument at the level of the landscape function.
\end{remark}
Let us provide a few comments on Theorem~\ref{tGrAnt}.

These estimates capture an effect of randomness as opposed to simple spectral gap considerations. Indeed, the spectrum of $-\triangle+(\eta+\lambda V)$ starts at $\eta$. 
 If $\omega$ were 
treated deterministically (or in other words, without taking an expectation on the left-hand side of \eqref{e.annealed-mass}), one would not be able to assure the exponential decay without dependence on $\eta$\footnote{If we do not take the expectation on the left-hand side of  \eqref{e.annealed-mass}, one needs $C_{\lambda,p}$ to be a random field depending on $x$ -- this would correspond to the quenched version of \eqref{e.annealed-mass}, which we also prove in this note.}.
Theorem~\ref{tGrAnt} is however not surprising:  it is certainly well-known to the experts, and should essentially follow from a post-processing of subtle localization results, such as \cite{MR2207021,MR2998830} in the continuous setting. It should also follow from \cite{MR1675027,MR2800903} for the discrete Laplacian on $\Z^d$, see also \cite{MR1717054} when the obstacles on $\Z^d$ are replaced by Poisson obstacles on $\R^d$ -- see also the early references \cite{MR496278,MR1836558} on these questions. It would however take 
 some
re-proving, such a result is not stated explicitly in any of the aforementioned works (nor would it be possible inside the spectrum), and we do not want to use either of these routes, in part because known treatments of Anderson localization still do not include all potentials covered in this note, and in part because that would make the methodology unnecessarily complicated. 

Here we present a somewhat more streamlined and self-contained argument, which also yields stronger results in our setting. Indeed, contrary to the aforementioned estimates on the Green function in papers devoted to the Anderson localization, we control high moments. The reason is that we only focus on the bottom of the spectrum, and in particular, we have positivity methods (such as Harnack's inequality) at our disposal. The proof we display below is short, has a PDE flavor, and seems to be new.
It is based on Agmon's positivity method \cite{MR745286}, which we combine with rank-one perturbation for $d\ge 3$ and with first-passage percolation in dimensions $d\le 2$. 

The scaling $\sqrt \lambda \wedge 1$ of the exponential decay with respect to $\lambda$ is optimal, that is, that the rate of the exponential decay does not depend on $\lambda$ when $\lambda\uparrow +\infty$ and that it is of order $\sqrt \lambda$ when $\lambda \le 1$. 
This has to do with the percolation assured in our setting. Indeed, given that $\varphi$ is supported in $B_{1/10}$, the set where $V$ is zero percolates.  For every finite $\lambda$ we have $G_{\lambda} \ge G_{\infty}$ by the maximum principle, where we denote by $G_\infty$ the limit of the Green function as $\lambda \uparrow +\infty$. It vanishes identically where $V$ is positive and it can be viewed as a Dirichlet Green function of the set $\{V=0\}$.  Since $G^{\infty}$ itself does not carry a dependence on $\lambda$  and does not decay faster than exponentially 
 by Harnack's inequality,
 the fact that $G_{\lambda} \ge G^{\infty}$ assures that the rate of decay of $G_{\lambda}$ cannot get faster as $\lambda \to \infty$. For the optimality of the scaling for $0\le \lambda \ll 1$, we refer to 
\cite{MR1717054,MR2800903}, which treat models to which our approach also applies.

\section{Proof of the exponential decay of the Green function}

To start, let us recall perturbative deterministic estimates on $G_\eta$.
By the maximum principle, 
\begin{equation}\label{e.bd-Greenlap-bis}
0\le G_\eta(x,y) \le \tilde G_\eta(x-y),
\end{equation}
where $\tilde G_\eta$ denotes the
Green function of the massive Laplacian $-\triangle+\eta$.
The latter satisfies the two-sided bound (with different constants $c_d$) 
\begin{equation}\label{e.bd-Greenlap}
\tilde G_\eta(y) \, \simeq \, 
\exp(-c_d\sqrt{\eta} |y|) \times \left\{
\begin{array}{lll}
\sqrt{\eta}^{-1}&:&d=1,\\
 \log(2+\frac{1}{\sqrt{\eta}|y|})&:&d=2,\\
|y|^{2-d}&:&d>2,
\end{array}
\right. 
\end{equation}
and its gradient satisfies
\begin{equation}\label{e.bd-GreenlapGrad}
|\nabla \tilde G_\eta(y)| \, \lesssim \, 
\exp(-\sqrt{\eta} |y|)  |y|^{1-d}.
\end{equation}
 Since $V$ is positive, 
we can apply Caccioppoli's estimate to $\nabla G_\eta$ as usual (say, on a ball $B_{|y|/2}(y)$) -- see, e.g., Lemma~4.1 in \cite{MR3815288},    and hence we also have a bound on $\nabla G_\eta$, however with 
a dependence on $\eta$ in dimensions $d=1,2$:
\begin{equation}\label{e.bd-GreenGrad}
|\nabla   G_\eta(y)| \, \lesssim \, 
\exp(-c_d\sqrt{\eta} |y|)  |y|^{1-d} \times \left\{
\begin{array}{lll}
\sqrt{\eta}^{-1}&:&d=1,\\
1+| \log \eta|&:&d=2,\\
1&:&d>2,
\end{array}
\right. 
\end{equation}
with a slightly different constant $c_d$ still depending on dimension only.
Note that the exponential decay of $G_\eta$ cannot be uniform with respect to $\eta>0$  as there could be arbitrarily large regions where $V$ is simply zero and $G_\eta$ should more or less look like $G$. As we will show below, it is however uniform after taking a statistical average (and, in fact, any finite moment), which is an effect of randomness.

\subsection{Agmon's positivity method}

The upcoming lemma is the starting point of Agmon's positivity method, cf.~\cite{MR745286}.
It is a variant of Caccioppoli's inequality when the test-function is chosen as the solution times an exponential weight.
\begin{lemma}\label{lem:Agmon}
Let $h:\R^d \to \R_+$ be a non-negative differentiable function, and $\chi:\R^d\to [0,1]$ be a smooth function with compact support that vanishes on 
 $Q:=[-\frac12,\frac12)^d$. Then 
 \begin{equation}\label{e.Agmon}
\tfrac12 \int \chi^2 e^{2h} |\nabla G_\eta|^2+\int \chi^2 e^{2h} G_\eta^2(\cdot,0) (\lambda V-4|\nabla h|^2) \,\le\, 4\int |\nabla \chi|^2 e^{2h} G_\eta^2(\cdot,0),
\end{equation}
where $\int$ denotes $\int_{\R^d}$.
\end{lemma}
\begin{proof}
Testing \eqref{defGr} with $\chi^2 e^{2h} G_\eta(\cdot,0)$ yields the identity
$$
\int \nabla (\chi^2 e^{2h} G_\eta(\cdot,0)) \cdot \nabla G_\eta(\cdot,0)+\int (\eta+\lambda V)\chi^2 e^{2h} G_\eta^2(\cdot,0)=0.
$$
Expanding $ \nabla (\chi^2 e^{2h} G_\eta)=\chi^2 e^{2h} \nabla G_\eta +2 G_\eta \chi^2 e^{2h} \nabla h+2G_\eta\chi e^{2h}\nabla \chi$, we obtain the claim by discarding the term with $\eta$, and using $ab \le \frac18 a^2+2b^2$
on each term to get 
$$
|(2 G_\eta \chi^2 e^{2h} \nabla h+2G_\eta\chi e^{2h}\nabla \chi)\cdot \nabla G_\eta |
\le \tfrac12\chi^2 e^{2h} |\nabla G_\eta|^2 + 4G_\eta^2  e^{2h}|\nabla \chi|^2 +4G_\eta^2\chi^2 e^{2h}|\nabla h|^2.
$$
\end{proof}
In order to infer exponential decay of $G_\eta$ from \eqref{e.Agmon}, one would need $h$ to be linear. 
However, $V$ is not uniformly bounded from below, and if $h$ is linear,  the integrand on the left-hand side of \eqref{e.Agmon} cannot be positive.
To circumvent this issue, in what follows, we shall combine  \eqref{e.Agmon} with Harnack's inequality and two probabilistic ingredients, depending on 
 the
dimension:
\begin{itemize}
\item For $d>2$, by taking the expectation of \eqref{e.Agmon} and using a rank-one perturbation argument, 
we may  
choose $h(x)=c|x|$  
and conclude directly;  
\item For $d\le 2$, this rank-one perturbation argument does not work (since we have no deterministic a priori bound on $G_\eta$ which is uniform in $\eta$). In this case (this holds in any dimension) up to technical details, we essentially choose for the Agmon function $h$  the chemical distance of first passage percolation, which we then control using the classical results \cite{MR876084} by Kesten.
\end{itemize}
By the maximum principle, for all $\lambda\le \lambda'$ we have $G_{\eta,\lambda'} \le G_{\eta,\lambda}$, so that
in the rest of this section we may assume without loss of generality that $\lambda \le 1$. This will make the technicalities simpler as in many instances the constants in Harnack's inequality will be uniform with respect to $\lambda$. 

\subsection{Rank-one perturbation for $d\ge 3$}\label{sd3}

We split the proof of \eqref{e.annealed-mass} into two steps.

\medskip

\step1  Proof of  \eqref{e.annealed-mass}, modulo the rank-one perturbation estimate \eqref{e.rank1}.\\
In what follows, we denote by $Q(z)$ the unit cube centered at $z\in \ZZ^d$. Since $\lambda \leq 1$ and the size of $Q(z)$ is 1, by Harnack's inequality (see for instance \cite[Theorem~8.20]{MR1814364}) for non-negative solutions ($(-\triangle+(\lambda V+\eta))G_\eta=0$ outside $Q$), there exists $C<\infty$ such that for all $z \in \Z^d$ with $z\ne 0$,
$\sup_{Q(z)} G_\eta^2(\cdot,0) \le C \inf_{Q(z)}G_\eta^2(\cdot,0)$.
Hence, \eqref{e.Agmon} entails
\begin{equation} \label{e.app-Agmon0}
\sum_{z\in \Z^d\setminus\{0\}} \sup_{Q(z)}G_\eta^2(\cdot,0) 
\int_{Q(z)} \chi^2 e^{2h}  (\tfrac1C \lambda V -|\nabla h|^2) \, \lesssim
\, \sum_{z\in \Z^d\setminus\{0\}} \sup_{Q(z)}G_\eta^2(\cdot,0)\int_{Q(z)} |\nabla \chi|^2 e^{2h}.
\end{equation}
Set now $G_*(z):=\expec{\sup_{Q(z)} G_\eta^2(\cdot,0)}^\frac12$.
Assume that $h$ and $\chi$ are deterministic.
By taking the expectation on both sides of the above, we obtain
\begin{equation} \label{e.exp-Agmon}
\sum_{z\in \Z^d\setminus\{0\}} \int_{Q(z)} \chi^2 e^{2h}  (\tfrac1C \expecm{\lambda V \sup_{Q(z)}G_\eta^2(\cdot,0)} -G_*^2(z)|\nabla h|^2) \,\le\, \sum_{z\in \Z^d\setminus\{0\}} G_*^2(z)\int_{Q(z)} |\nabla \chi|^2 e^{2h}.
\end{equation}
Assume now that there exists $C<\infty$ such that 
\begin{equation}\label{e.rank1}
\expecm{\lambda V \sup_{Q(z)}G_\eta^2(\cdot,0) } \ge \tfrac1C \lambda   G_*^2(z) \varphi(\cdot-z).
\end{equation}
Combined with \eqref{e.rank1}, \eqref{e.exp-Agmon} then takes the form
\begin{equation}\label{e.exp-Agmon2}
\sum_{z\in \Z^d\setminus\{0\}} G_*^2(z)\int_{Q(z)} \chi^2 e^{2h}  (\tfrac1{C^2} \lambda  \varphi(\cdot-z) -|\nabla h|^2) \,\le\, \sum_{z\in \Z^d\setminus\{0\}} G_*^2(z)\int_{Q(z)} |\nabla \chi|^2 e^{2h}.
\end{equation}

At this point we basically want to choose $h$ as a multiple of $\sqrt \lambda |x|$ and $\chi$ equal to 0 near zero and 1 away from a ball of radius 2 so that the integral on the left-hand side is roughly $e^{c\sqrt \lambda |z|}$ and the right hand side is confined to a ball of radius 2 and bounded by a constant. This would give the desired exponential decay of the Green function. However, the formalities are a bit more technical than that, and we follow \cite[Appendix~B]{MR2826466}.

For all $n,m$, define $\chi_n$ and $h_m$ via
\begin{eqnarray*}
&&\chi_n|_{Q_1}\equiv 0, \quad \chi_n|_{Q_n \setminus Q_2}\equiv 1, \quad \chi_n|_{\R^d \setminus Q_{n+1}}\equiv 0, \quad |\nabla \chi_n|\le 2,\\
&& h_m(x)=\mu |x|_\infty \text{ for }x\in Q_{2m}, \ h_m(x)=\mu m\text{ for }x\in  \R^d \setminus Q_{2m},
\end{eqnarray*}
where $Q_n$ is a cube centered at zero of the sidelength $n$ and 
$\mu=c\sqrt \lambda$ with $c$ to be chosen later. With these choices, for any finite $n,m$ the inequality \eqref{e.exp-Agmon2} is satisfied with $h=h_m$ and $\chi=\chi_n$. Fixing $m$ and using the decay of the Green function given by $G_\eta \le \tilde G_\eta$, and \eqref{e.bd-Greenlap} and boundedness of $h_m$, we can pass to the limit $n\uparrow +\infty$  by the dominated convergence theorem. Combining all the terms with $z\in \Z^d, |z|_\infty \le 2$, on the right-hand side of the inequality, this gives
\begin{multline}\label{e.exp-Agmon3}
\sum_{z\in \Z^d,\, |z|_\infty>2} G_*^2(z)\int_{Q(z)} e^{2h_m}  (\tfrac1{C^2}  \lambda  \varphi(\cdot-z) -  \mu^2) \,\\
\le\, \sum_{z\in \Z^d\setminus\{0\}, \,|z|_\infty \le 2} G_*^2(z)\int_{Q(z)} (\tfrac1{C^2}  \lambda   \varphi(\cdot-z) + \mu^2+4 )e^{2h_m}.
\end{multline}
We now choose $\mu$. By definition of $h_m$ we have (up to redefining $C$)
$$
\int_{Q(z)} e^{2h_m}  \tfrac1{C^2}  \varphi(\cdot-z) \ge \Big(\int_{Q} \varphi   \Big)  \int_{Q(z)}  e^{2h_m}    e^{-2\mu}\tfrac1{C^2}  =\int_{Q(z)}  e^{2h_m}    e^{-2\mu}\tfrac1{C} ,
$$
so that for the choice $\mu = c \sqrt{\lambda}  $ with a suitable $c$ and using the deterministic bound
$G_\eta(x,0)\le C|x|^{2-d}$ 
 (recall that $d\geq 3$), \eqref{e.exp-Agmon3} turns into
\begin{equation*} 
\sum_{z\in \Z^d,\, |z|_\infty>2} G_*^2(z)\tfrac1C  {\lambda}  \int_{Q(z)} e^{2h_m}   \,
\le\, C % .
\end{equation*}
(we only care about small $\lambda$).
By the monotone convergence theorem in the limit $m\uparrow +\infty$, this entails for some constants $C\ge 1 \ge c>0$ independent of $\lambda$
$$
G_*(z) \,\le\, C \frac1{\sqrt{\lambda}} \exp(-c  \sqrt{\lambda} |z|) \quad \mbox{for all $z\ne 0$}.
$$
Combining this with Harnack's inequality and the deterministic bound $G_\eta(x,y)\le C |x-y|^{2-d}$, then yields 
$$\expec{(\int_{Q(x)} G_\eta(\cdot,0))^2}^\frac12 \lesssim \Big( \frac1{\sqrt{\lambda} } \exp(-c \sqrt{\lambda} |x|)\Big)\wedge 1, \quad\mbox{for all $x\in \R^d$,}$$
and once again using Harnack's inequality, the same deterministic bound, and interpolation, we establish
\begin{equation*}
0 \le \expec{(\int_Q G_\eta(x,y+y')dy')^{2p}}^\frac1{2p} \le  \,C \Big(  \lambda^{-\frac1{2p}}   \exp(-\tfrac cp  \sqrt{\lambda} |x-y|)\Big)\wedge 1, \,\,\mbox{for all $p\ge 1$, $x,y\in \R^d$,}
\end{equation*}
and ultimately, \eqref{e.annealed-mass}.

\medskip

\step2 Proof of \eqref{e.rank1} by rank-one perturbation.\\
For all $z \in \Z^d$, denote by $G_\eta^{z,-}$ the Green function
for the  potential  $V^{z,-}$ associated with $\omega^{z,-}$
defined by $\omega^{z,-}_j=\omega_j$ if $j\ne z$ and $\omega^{z,-}_z=1$. By the maximum principle we have $0\le G_\eta^{z,-} \le G_\eta$, so that 
\begin{equation*}
\expecm{V \sup_{Q(z)}G_\eta^2(\cdot,0) } \ge \expecm{V \sup_{Q(z)}(G_\eta^{z,-}(\cdot,0))^2 } .
\end{equation*}
On $Q(z)$, $V$ only depends on $\omega_z$ whereas $G_\eta^{z,-}(\cdot,0)$ is independent on $\omega_z$ by construction, so that by independence, 
$$
 \expecm{V \sup_{Q(z)}(G_\eta^{z,-}(\cdot,0))^2 }=\varphi(\cdot-z) \expec{\omega} \expecm{\sup_{Q(z)}(G_\eta^{z,-}(\cdot,0))^2 }.
$$
To prove \eqref{e.rank1}, it remains to argue that $\expecm{\sup_{Q(z)}(G_\eta^{z,-}(\cdot,0))^2 }\gtrsim   \expecm{\sup_{Q(z)}(G_\eta (\cdot,0))^2 }$. To this aim, we 
 notice 
that the function $G_\eta-G_\eta^{z,-}$ satisfies the equation
\begin{equation}\label{e.eq-diff}
-\triangle (G_\eta -G_\eta^{z,-})+(\lambda V+\eta) (G_\eta -G_\eta^{z,-}) = \lambda (1-\omega_z) \varphi(\cdot -z)   G^{z,-}_\eta,
\end{equation}
so that we have the Green representation formula 
\begin{equation}\label{e.GF-1}
 (G_\eta-G_\eta^{z,-})(x,0) = \lambda \int_{Q(z)} G_\eta(x,x')(1- \omega_{z})\varphi(x'-z)G_\eta^{z,-}(x',0)dx'.
\end{equation}
We combine \eqref{e.GF-1} with the deterministic bound $0\le G_\eta(y,y')\le C |y'-y|^{2-d}$ to the effect that for all $x \in Q(z)$,
\begin{equation*} 
0\le (G_\eta -G_\eta^{z,-})(x) \le C \lambda \sup_{Q(z)}G_\eta^{z,-},
\end{equation*}
from which we deduce (using $\lambda \le 1$)
$$
\sup_{Q(z)}G_\eta^{z,-} \ge \tfrac1{C}  \sup_{Q(z)} G_\eta ,
$$
and therefore \eqref{e.rank1}.

\subsection{Agmon function and first-passage percolation for $d\ge 1$}

In low dimensions $d=1,2$, we do not have the deterministic and uniform bound $G_\eta \lesssim 1$.
By a percolation argument, 
 we will first check 
that for all $p\ge 1$, $\sup_z \expec{(\int_{Q(z)} G_\eta^p(\cdot,0))^p}^\frac1p \lesssim_p 1$, which we refer to as anchoring, cf.~Lemma~\ref{lem:anchor} below. This is however not sufficient to run (optimally) the rank-one perturbation argument displayed above for $d\ge 3$.
Instead, we use a quenched approach, and use as a function $h$ in Lemma~\ref{lem:Agmon} the chemical distance in first passage percolation. Using the classical bounds  \cite{MR876084} by Kesten on first passage percolation, we then post-process the quenched bounds into annealed bounds on the Green function.

\subsubsection{Anchoring by percolation}
\begin{lemma}[Anchoring]\label{lem:anchor}
There exists a stationary random field $r_*$ on $\R^d\times \R^d$ such that for all $\eta>0$ and all $x,y\in \R^d$
we have 
$$
0\le \int_{Q}G_\eta(x,y')dy'  \le r_*(x,y) +\left\{
\begin{array}{rcl}
d=1&:& C \lambda^{-1} ,
\\
d=2&:& C \lambda^{-1} ,
\\
d>2&:&1,
\end{array}
\right.
$$
and for all $p\ge 1$,
\begin{equation}\label{e.anchoring}
\sup_{x,y} \expec{r_*^p(x,y)}^\frac1p \leq  C 
\left\{
\begin{array}{rcl}
d=1&:& p ,
\\
d=2&:& \log(p+1) ,
\\
d>2&:&1.
\end{array}
\right.
\end{equation}
\end{lemma}
In dimensions $d>2$, \eqref{e.anchoring} follows from the deterministic bound $G_\eta(x,y)\lesssim |x-y|^{2-d}$ on the Green function. In dimensions $d\le 2$, \eqref{e.anchoring} 
will follow 
from integrating the defining equation for $G_\eta$ in combination with the maximum principle and a percolation argument (only needed for $d=2$).
\begin{proof}[Proof of Lemma~\ref{lem:anchor}]
It remains to treat $d=1,2$.
We integrate the equation for $G_\eta$ on $\R^d$, which yields using exponential decay
\begin{equation}\label{e.int-eq}
\int  (\lambda V+ \eta)(y) G_\eta(0,y)dy=1.
\end{equation}
We start with $d=1$, and then turn to $d=2$.

\medskip

\step1 Dimension $d=1$.\\
Let $y' \in \Z$, and let $y_1>y'$ (respectively, $y_{-1}<y'$) be the closest integer to $y'$ such that $\omega_{y_1}=1$ (respectively, $\omega_{y_{-1}}=1$). Then, by \eqref{e.int-eq} and Harnack's inequality (say, on $Q(y_1)$), $0\le G_\eta(0,y_1) \le C\lambda^{-1}$ (respectively, $0\le G_\eta(0,y_{-1})\le C\lambda^{-1}$). By the maximum principle, we thus have $0\le G_\eta(0,y) \le g(y)$, where $g$ solves
$$
-g''(y)=\delta(y) \text{ on }(y_{-1},y_1), \quad g(y_{-1})=G_\eta(0,y_{-1}), \quad g(y_1)=G_\eta(0,y_1). 
$$
The function $g$ can be decomposed as a sum of the Laplace's Dirichlet Green function on $(y_{-1}, y_1)$ and a harmonic function with data $G_\eta(0,y_{-1})$, $G_\eta(0,y_1)$ at the endpoints. The latter is a linear function maximized 
 for $G_\eta(0,y_{-1})$ or $G_\eta(0,y_1)$,
and the former is 
 at most $y_{1}-y_{-1}$, 
so that all in all we have 
$$0\le g(y') \le G_\eta(0,y_{-1})+G_\eta(0,y_1)+ y_{1}-y_{-1},$$
so that
\begin{equation} \label{e.added} 
G_\eta(0,y') \,\le\,C\lambda^{-1}+( y_1-y_{-1}).
\end{equation}
We then set $r_*(0,y'):=  y_1-y_{-1}$, and it remains to notice that for all $n>0$, $\bP[y_1-y'=n]  = \bP[y'-y_{-1}=n] \le C \exp(-\frac1C n)$, to the effect that $\expec{r_*(0,y')^p}^\frac1p \le C p$ for all $p\ge 1$.
To pass from integers $y'$ to real numbers, we use Harnack's inequality.
Since $0$ plays no role in the argument, this concludes  the proof for $d=1$.

\medskip

\step2 Dimension $d=2$.
\noindent
The idea is the same as for $d=1$. 
Set 
$g_\eta(y)=\int_Q G_\eta (x,y)dx$,  so that, by \eqref{e.int-eq}, 
\begin{equation}\label{e.int-eq-b}
\int  (\lambda V+ \eta)(y) g_\eta(y) dy=1 %, 
\end{equation}
and, by \eqref{defGr}, 
\begin{equation}\label{defGr-b}
(-\triangle_y +(\lambda V+\eta))g_\eta(y) =\mathds 1_Q (y).
\end{equation}
Much as for $d=1$, given $y'\in \R^d$, we want to construct a set  $C(y')$ containing $y'$, and such that on its boundary $\partial C(y')$  we control the
function $g$. Because of  \eqref{e.int-eq-b},
we know that on any cube where $V$ is larger than a constant $c_0$ the function $g \lesssim 1/(\lambda c_0)$. The problem is that  these cubes might fail to 
percolate and form connected sets, much less the sets at a controlled distance from $y'$. The remedy is to enlarge these cubes by a fixed multiple depending on the probability law only to ensure percolation and preserve the control of the Green function using the Harnack inequality. 

For future reference, we consider general dimension $d$. To this end, consider the partition $\calQ_k$ of $\R^d$ into cubes $Q_{k}+(2^k \Z)^d$, where $Q_{k}=[-2^{k-1},2^{k-1})^d$ and $k\in \N$ will be chosen below (depending on the probability law of $V$ and independently of $\lambda$). Ultimately, $2^k$ will be the multiple alluded to before. 
Call $\calT_k$ the graph associated with $(2^k \Z)^d$ via $e \in  \calT_k$ iff $e=(z,z')$ with $z,z' \in (2^k \Z)^d$, 
$|z-z'|=2^k$ (the Euclidean distance).
For each edge $e=(z,z')$ of $\calT_k$ we associate the open cube $Q'_e$ centered at $\frac{z+z'}2$ and of sidelength $2^{k-1}$ (to the effect that those cubes are disjoint).
We now define conductivities $(\xi_e)_{e \in \calT_k}$ as follows. Let $\gamma>0$ be such that $\bP[\omega \ge \gamma]>0$.
Set $\xi_e =1$ 
if $\max_{j \in \Z^d \cap Q'_e} \omega_j \ge \gamma$, and $\xi_e=0$ otherwise. 
The sequence $\{\xi_e\}_{e \in \calT_k}$ has iid Bernoulli entries with the probability that $\xi_e=0$ equal to 
$$p=\bP[\omega < \gamma]^{\sharp (Q'_{e} \cap \Z^d)} \leq \bP[\omega < \gamma]^{\frac 1C 2^k-1} $$ 
as  $\sharp (Q'_{e} \cap \Z^d) \ge \frac1C 2^k-1$. We then choose $k$ large enough (and depending only on the probability law $F$) such that $p<\frac12$, so that the edges $e$ such that $\xi_e=1$ percolate
 almost surely. 

Denote by $\mathcal E$ the unique percolation cluster (a set of edges) given by the procedure above and by $\calZ$ the set of vertices of $2^k \Z^d$ forming the endpoints of these edges.
Notice that Harnack's inequality yields on each cube $Q_k(z)=z+Q_k$ with $z\in \calZ$
\begin{equation}\label{e.Harnack-per}
\sup_{Q_k(z)} g_\eta \,\le\, C_k \Bigl(\inf_{Q_k(z)} g_\eta + 1\Bigr)
\end{equation}
(with a constant depending on $k$ but independent of $\lambda$, which is possible due to our assumption $\lambda \le 1$). If $y' \in \bigcup_{z \in \calZ} Q_k(z) $, then by \eqref{e.int-eq-b} and \eqref{e.Harnack-per}, $g_\eta(y') \le C \lambda^{-1}$.
If $y' \notin \bigcup_{z \in \calZ} Q_k(z)$, hence $y' \in Q_k(z_y)$ for some $z_y \in 2^k\Z^d \setminus \calZ$,
we call $\calC_d(y')$ the connected component of  $2^k\Z^d \setminus \calZ$ 
containing $z_y$ -- where connectedness is understood in the discrete sense, 
that is, two points $x'$ and $x''$ of $2^k \Z^d$ are connected iff $|x'-x''|=2^k$ (again, the Euclidean distance). 
We then set $\calC(y')=\cup_{z \in \calC_d(y')} Q_k(z)$, which is our analogue of the set $(y_{-1}, y_1)\ni y'$ from the $d=1$ case.

On $\partial \calC(y')$, \eqref{e.int-eq-b} and \eqref{e.Harnack-per} imply that $g_\eta \le C\lambda^{-1}$. Similarly to the case $d=1$, we then control $g_\eta$ inside $\calC(y')$ by $g$, the function defined by
$$
-\triangle g= \mathds1_Q \text{ in }C(y'), \quad g=g_\eta \text{ on }\partial \calC(y'),
$$
where $Q$ is again the unit cube centered at $0$ that was used to define $g_\eta$ in \eqref{defGr-b}.
By the maximum principle %, 
and the fact that $\log(|x|)$ is the fundamental solution of the Laplacian,
$0\le \sup_{\calC(y')} g \le \sup_{\partial \calC(y')} g_\eta+C \log(\diam \calC(y'))$, so that $$g_\eta(y') \le  C\lambda^{-1}+C \log(\diam \calC(y'))$$ % .
(analogously to \eqref{e.added} above).
The conclusion follows from subcritical Bernoulli percolation estimates in the form $\bP[\diam \calC(y') \ge n] \lesssim \exp(-\frac1C n)$ (see \cite{MR874906}) and the fact that $G_\eta(x,y)=G_\eta(y,x)$. 
\end{proof}

\subsubsection{Agmon distance, exponential decay, and first-passage percolation}

We now prove Theorem~\ref{tGrAnt} for $d=1,2$. The argument is a probabilistic upgrade of the Agmon exponential decay estimates that we used for the case $d>2$. 
 When $d>2$ we were able to 
provide deterministic anchoring of the Green function and this furnished a more straightforward argument. For low dimensions, the proof is resonating but more delicate as the probabilistic argument has to be intertwined with the elliptic considerations from \S\ref{sd3}. 
The upcoming proof holds in any dimension $d\ge 1$, which we all treat at once.
The strategy relies on the exponential decay measured in the chemical distance to the origin using Agmon's positivity method, followed by a ``change of variables'' using first passage percolation estimates to revert back to the Euclidean distance.

\medskip

\step{1} The chemical distance. \\
We use the notation of Step~2 of the proof of Lemma~\ref{lem:anchor}.
In this case, $\calE$ is the set of edges of $(2^k)\Z^d$, and $\{\xi_e\}_{e\in \calE}$ denotes a set of weights on edges.
Recall that if $\xi_{[x,x']}=1$, then there exists $z\in Q_{k+1}(x) \cap Q_{k+1}(x')\cap \Z^d$ such that $w_z>\gamma>0$.
We say that $\pi=(x_0,\dots,x_{N+1})$ is a path of length $N+1$ in $(2^k)\Z^d$ between $x_0$ and $x_{N+1}$ if for all $0\le j\le N$ we have $|x_{j+1}-x_j|=2^k$. The path is self-avoiding if $x_j\ne x_i$ for all $i\ne j$.
To any path $\pi$ we assign the weight 
$$
w(\pi)=\sum_{j=0}^N \xi_{[x_j,x_{j+1}]}.
$$
We then define the chemical distance $d_\chi:2^k\Z^d \times 2^k\Z^d \to \R$ via
$$
d_\chi(x,x') = \inf\Big\{ w(\pi)\,:\, \pi \text{ path between $x$ and $x'$}\Big\}.
$$
Finally, we define the function $h:2^k\Z^d\to \R_+, x \mapsto d_\chi(x,0)$, which we extend on $\R^d$ by using a continuous piecewise affine interpolation (say, on a tetrahedral tessellation of $\R^d$ associated with $2^k \Z^d$).
Since for all $x,x'\in 2^k \Z^d$ with $|x-x'|=2^k$, $|d_\chi(x,0)-d_\chi(x',0)| \in \{0,1\}$, this extension is $2^{-k}$-Lipschitz.
\\

\step{2} Chemical exponential decay via Agmon positivity method.\\
Set $(2^k \Z^d)^*:=\{x \in 2^k\Z^d \,|\, |x|\ge 2^{k+2}\}$, and define $G_{\eta,*}=\sum_{x\in (2^k \Z^d)^*} (\sup_{Q_{k+1}(x)} G_\eta) \mathds1_{Q_{k}(x)}$. 
By Harnack's inequality on each cube $Q_{k+1}(x)$ (which do not contain the origin), we have
\begin{equation*}\label{e.Harnack-per+}
\sup_{Q_{k+1}(x)} G_\eta \,\le\, C_k  \inf_{Q_{k+1}(x)}G_\eta 
\end{equation*}
(with a constant depending on $k$ but independent of $\lambda \le 1$).
Proceeding as in Step~1 in Subsection~\ref{sd3}
and keeping the gradient term in \eqref{e.Agmon}, we obtain for all $\mu>0$ and $m\ge 1$ by monotone and dominated convergence
\begin{multline*}
\int e^{2(h \wedge m)\mu} |\nabla G_\eta|^2+\sum_{x \in (2^k \Z^d)^*}  e^{2 (h(x)\wedge m) \mu} \Big( \int_{Q_{k+1}(x)} G_{\eta,*}^2   \tfrac{e^{-2\mu}}{C^2}  \lambda V
  - \int_{Q_{k}(x)}  G_{\eta,*}^2 Ce^\mu \mu^2|\nabla h|^2\Big)
\\
\le\, C  \int_{Q_{k+1}(0)}  G_\eta^2(\cdot,0).
\end{multline*}
If for some $x\in 2^k \Z^d$, $\nabla h|_{Q_k(x)} \not\equiv 0$, then there exists $z \in Q_{k+1}(x) \cap \Z^d$ such that $\omega_z = 1$. Combined with the $2^{-k}$-Lipschitz property of $h$, this entails with the choice $\mu=c \sqrt \lambda$
\begin{equation*}
\int_{Q_{k+1}(x)} G_{\eta,*}^2   \tfrac{e^{-2\mu}}{C^2}  \lambda V
  - \int_{Q_{k}(x)}   G_{\eta,*}^2 Ce^\mu \mu^2|\nabla h|^2 \ge 0,
\end{equation*}
so that the above takes the form
\begin{equation} \label{e.Agmon-chemical}
\int e^{2c\sqrt \lambda h} |\nabla G_\eta|^2\, \le\, C  \int_{Q_{k+1}(0)}  G_\eta^2(\cdot,0).
\end{equation}

\medskip

\step{3} Control of the chemical distance.\\
Since $p$ is chosen small enough, we are in the regime of super-critical percolation, and \cite[Proposition~(5.8)]{MR876084} reads: There exists $0\le c \le 1 \le C$ such that for all $R\ge 1$,
\begin{multline*}
\bP[\exists \text{ a self-avoiding path }\pi \\
\text{ starting at 0 of length at least $2^kR$ such that }w(\pi)\le cR] \le C \exp(-cR).
\end{multline*}
This implies in particular that for all $R\ge 1$
$$
\bP[d_\chi(0,2^k\Z^d \setminus [-2^{k-1}R,2^{k-1} R)^d) \le cR ] \le C\exp(-cR),
$$
so that there exist $c,C$ such that for all $l\in \N$,
\begin{equation}\label{e.Kesten}
\bP[\inf_{S_l^+} h \le c 2^l] \le C \exp(-c2^l).
\end{equation}

\medskip

\step4 Conclusion.\\
It remains to post-process \eqref{e.Agmon-chemical} and \eqref{e.Kesten}.
Set $S_l:=\{2^l <|x| \le 2^{l+1}\}$ and $S_l^+:=\{2^l <|x| \le 2^{l+2}\}$.
On the one hand, since $G_\eta$ decays exponentially fast by \eqref{e.bd-Greenlap-bis} \& \eqref{e.bd-Greenlap}, we have the telescopic identity
\[
\fint_{S_j} G_\eta = \sum_{l=j}^\infty \Big(\fint_{S_l}G_\eta-\fint_{S_{l+1}}G_\eta\Big).
\]
On the other hand, by the fundamental theorem of calculus and Cauchy-Schwarz' inequality, and \eqref{e.Agmon-chemical},
\begin{align*}
\Big|\fint_{S_l}G_\eta-\fint_{S_{l+1}}G_\eta\Big| \,\le& \,C 2^l \Big(\fint_{S_l^+} e^{2c\sqrt \lambda h} |\nabla G_\eta|^2 \Big)^\frac12 \Big(\fint_{S_l^+} e^{-2c\sqrt \lambda h}\Big)^\frac12
\\
\le &\,C\Big(\int_{Q_{k+1}(0)}  G_\eta^2(\cdot,0)\Big)^\frac12 (2^l)^{1-\frac d2}\Big(\fint_{S_l^+} e^{-2c\sqrt \lambda h}\Big)^\frac12 .
\end{align*}
Hence, by Harnack's inequality,
\[
\sup_{S_j} G_\eta \,\le \,C   \Big(\int_{Q_{k+1}(0)}  G_\eta^2(\cdot,0)\Big)^\frac12 \sum_{l=j}^\infty (2^l)^{1-\frac d2}\Big(\fint_{S_l^+} e^{-2c\sqrt \lambda h}\Big)^\frac12 .
\]
Combined with the following consequence of \eqref{e.Kesten} 
\[
\expec{\fint_{S_l^+} e^{-2 c\sqrt \lambda h} }^\frac12 \le \Big( \expec{e^{-c\sqrt \lambda 2^l}\mathds1_{\inf_{S_l^+} h > c2^l}}+\bP[\inf_{S_l^+}h \le c2^l]\Big)^\frac12 \,\le\, C e^{- c \sqrt \lambda 2^l} ,
\]
and with Lemma~\ref{lem:anchor},
this implies by Cauchy-Schwarz' inequality in probability
\begin{align*}
\expec{\sup_{S_j} G_{\eta}} \le& \,   
C   \expec{\int_{Q_{k+1}(0)}  G_\eta^2(\cdot,0)}^\frac12 \sum_{l=j}^\infty (2^l)^{1-\frac d2}\expec{\fint_{S_l^+} e^{-2c\sqrt \lambda h}}^\frac12
\\
\le &\, C\sum_{l=j}^\infty (2^l)^{1-\frac d2}   e^{- c \sqrt \lambda 2^l} \times \left\{
\begin{array}{rcl}
d=1&:&  \lambda^{-1} ,
\\
d=2&:&  \lambda^{-1} ,
\\
d>2&:&1,
\end{array}
\right.
\\
\le& \,C e^{- c \sqrt \lambda 2^j} \times \left\{
\begin{array}{rcl}
d=1&:&  \lambda^{-1} ,
\\
d=2&:&  \lambda^{-1} ,
\\
d>2&:&1,
\end{array}
\right.
\end{align*}
and \eqref{e.annealed-mass} follows by interpolation with Lemma~\ref{lem:anchor} again.

\section{Application to the landscape function}\label{sec:landscape}

\subsection{Existence and control of the moments: proof of Theorem~\ref{tLand}}\label{sec:0}

Let $0<\eta <1$ and consider the modified landscape equation 
\begin{equation}\label{e.landscape+}
-\triangle u_\eta + \lambda (V+\eta)u_\eta =1.
\end{equation}
By the exponentially weighted energy estimate
$$
\int_{\R^d} \exp(-c\sqrt  \eta |x|) (|\nabla u_\eta|^2+(\lambda V+\eta)u_\eta^2)(x) dx \lesssim
\frac1\eta \int_{\R^d}   \exp(-c\sqrt \eta |x|) dx
$$
(which we obtain by testing the equation with $x \mapsto \exp(-c\sqrt \eta |x|)u_\eta$ for $c>0$ small enough), this equation is well-posed in the space $H^1_{\uloc}(\R^d)=\{v \in H^1_\loc(\R^d) \,|\, \sup_{x\in \R^d} \int_{B(x)} v^2+|\nabla v|^2 <+\infty\}$. In addition we have the
representation formula for all $x\in \R^d$
\begin{equation}\label{e.ueta}
0 < u_\eta(x)\,=\,\int  G_\eta(x,y)dy \, \le \, \int  \tilde G_\eta(y)dy,
\end{equation}
and $u_\eta$ is bounded on $\R^d$ (with the bound depending on $\eta$).
By uniqueness  in the space above, 
$u_\eta$ is stationary, and
 we now check that
the energy estimate for \eqref{e.landscape+},  
 in combination with the ergodic theorem,  
 yields
\begin{equation}\label{e.bd-grad-u}
\expec{\int_Q |\nabla u_\eta|^2} \le \expec{\int_Q u_\eta}.
\end{equation}
Indeed, let $\chi$ be a smooth non-negative, compactly supported function of mass unity.
For all $R\ge 1$, set $\chi_R := R^{-d} \chi(\cdot /R)$.
Then for all $R$ we have by testing  \eqref{e.landscape+}  with $\chi_R u_\eta$ and discarding the non-negative term in the right hand side
\begin{eqnarray*}
\int  \chi_R u_\eta &=& \int \nabla (\chi_Ru_\eta)\cdot \nabla u_\eta+ ( \lambda V+\eta) \chi_R u_\eta^2
\\
&\ge & \int \chi_R  |\nabla u_\eta|^2
 - \tfrac 1 R \int  R^{-d}u_\eta|\nabla \chi|(\frac \cdot R) |\nabla u_\eta|.
\end{eqnarray*}
Since $u_\eta \in H^1_{\uloc}(\R^d)$, the second right-hand side term vanishes in the limit $R\uparrow +\infty$, and \eqref{e.bd-grad-u} follows by the ergodic theorem in the form
$$
\lim_{R\uparrow +\infty} \int  \chi_R u_\eta = \expec{\int_Q u_\eta},
\quad \lim_{R\uparrow +\infty} \int R^{-d}\chi(\frac\cdot R) |\nabla u_\eta|^2 =\expec{\int_Q |\nabla u_\eta|^2}.
$$
Since $u_\eta$ is non-negative and uniformly bounded,
 for all $p\ge 1$ we have $\sup_{x} \expec{ u_\eta(x)^p} <\infty$. 
 We wish to pass to the limit in this bound as $\eta \downarrow 0$, and therefore need uniformity with respect to $\eta$. 
To this end, we appeal to Theorem~\ref{tGrAnt}.
Taking the $p$-th power and the expectation of \eqref{e.ueta}, we have by Minkowski's inequality
for all $x\in \R^d$
\begin{eqnarray*}
\expec{  u_\eta (x)^p}^\frac1p& =&\left(\expec{\Big( \int G_\eta(x,y )dy\Big)^p}\right)^\frac1p
\\
&=&\left(\expec{\Big(\sum_{z\in \Z^d}\int_{ Q} G_\eta(x,z+z') dz'\Big)^p}\right)^\frac1p
\\
&\le & \sum_{z\in \Z^d} \left(\expec{\Big( \int_{ Q} G_\eta(x,z+z')dz'\Big)^p}\right)^\frac1p,
\end{eqnarray*}
from which we conclude that  $\expec{  u_\eta(x)^p}^\frac1p \,\lesssim_{p,\lambda} \, 1$.
By elliptic regularity (Moser estimates - cf., e.g., \cite[Theorem~8.22]{MR1814364}), 
$$\sup_Q u_\eta \lesssim 1+\int_{2Q} u_\eta,$$
so that the above yields for all $x\in \R^d$,
\begin{equation} \label{e.land-bound-mu}
\expec{(\sup_Qu_\eta)^p}^\frac1p\,\lesssim_{p,\lambda}\, 1.
\end{equation}

This furnishes the desired moment bounds \eqref{eq.mom} for $u_\eta$ in place of $u$, and it remains to pass to the limit. 
In particular, used for $p=1$ and $p=2$, and combined with \eqref{e.bd-grad-u}, this entails that $u_\eta$ is  bounded in $L^2(\Omega,H^1(Q))$ 
 (where $\Omega$ is our probability set of events $\omega$) independently of $\eta>0$. 
Hence, by the Banach-Alaoglu theorem, there is a subsequence converging weakly towards a function $u \in L^2(\Omega,H^1(Q))$ (which we may extend by stationarity as a function of $L^2(\Omega,H^1_\loc(\R^d))$) that satisfies the same moments bounds as $u_\eta$.
It remains to show that $u$ solves \eqref{e.landscape}.
Let $\zeta \in L^2(\Omega)$ and $\chi \in C^\infty_c(\R^d)$. By testing 
\eqref{e.landscape+} with $\chi$, multiplying by $\zeta$, and taking the expectation,
we obtain
$$
\expec{\zeta\Big(\int  \nabla \chi \cdot \nabla u_\eta+(\lambda V+\eta)\chi u_\eta - \chi \Big)}=0.
$$
Taking the limit $\eta \downarrow 0$ (along the subsequence) yields
$$
\expec{\zeta\Big(\int  \nabla \chi \cdot \nabla u +\lambda V\chi u - \chi \Big)}=0.
$$
Since $C^\infty_c(\R^d)$ is separable and $\zeta$ is arbitrary, this implies that almost surely 
we have for all $\chi \in C^\infty_c(\R^d)$,
$$
\int  \nabla \chi \cdot \nabla u+\lambda V \chi u = \int \chi,
$$
and hence $u$ is a distributional solution of \eqref{e.landscape} almost surely.
Uniqueness is standard and follows the proof of \eqref{e.bd-grad-u} (we simply 
use $u\in L^2(\Omega,H^1(Q))$ rather than $u_\eta \in H^1_\uloc(\R^d)$ almost surely).
 Indeed let $u_1$ and $u_2$ be two stationary solutions in $L^2(\Omega,H^1(Q))$, and let $\chi$ be a smooth non-negative, compactly supported function of mass unity.
For all $R\ge 1$, set $\chi_R := R^{-d} \chi(\cdot /R)$.
Then for all $R$ we have by testing  \eqref{e.landscape}  with $\chi_R(u_1-u_2)$
\begin{eqnarray*}
0&=&\int \nabla (\chi_R(u_1-u_2))\cdot \nabla (u_1-u_2)+ \lambda V \chi_R(u_1-u_2)^2
\\
&\ge & \int  \chi_R |\nabla (u_1-u_2)|^2
 - \frac1 R \int  R^{-d}|u_1-u_2||\nabla \chi|(\frac \cdot R) |\nabla (u_1-u_2)|.
\end{eqnarray*}
Since $u_1,u_2 \in L^2(\Omega,H^1_\loc(\R^d))$ are stationary, taking the limit $R \uparrow +\infty$ yields by the ergodic theorem
$$
\expec{\int_Q |\nabla (u_1-u_2)|^2}=0,
$$
so that $\nabla u_1=\nabla u_2$. This implies in turn that $u_1=u_2+K$ for some random constant $K$. Using the zero order term in the above in form of $0=\int \lambda V \chi_R K^2$ for $R\gg 1$ large enough so that $V|_{\supp \chi_R} \not \equiv 0$ (which holds true almost surely), we deduce that $K=0$, and therefore $u_1=u_2$. 
  
\medskip

We conclude this subsection by proving that $u_\eta$ converges strongly to $u$ in $L^2(\Omega,W^{1,\infty}(Q))$ (this will be needed to control covariances).
The difference $u_\eta-u$ satisfies the equation
$$
-\triangle (u_\eta-u)+(\lambda V +\eta)(u_\eta-u) = \eta u,
$$
so that 
$$
(u_\eta-u)(x)= -\eta \int  G_\eta(x,y) u(y)dy.
$$
Hence, by the triangle inequality followed by Cauchy-Schwarz' inequality
$$
\expec{  (u_\eta-u)(x)^2}^\frac12  \,\le\, \eta \int  \expec{(\int_{Q}G_\eta(x,y+y')dy')^4}^\frac14 \expec{\sup_Q u(y+\cdot)^4}^\frac14 dy.
$$
By Theorems~\ref{tLand} and~\ref{tGrAnt}, the right-hand side is of order $\eta$.
Combined with elliptic regularity in the form
\begin{equation}\label{e.unif-grad}
\sup_Q |\nabla (u-u_\eta)|^2+\sup_Q |u-u_\eta|^2 \,\lesssim\, \int_{2Q} (u-u_\eta)^2+\eta\sup_{2Q}  u^2,
\end{equation}
and with Theorem~\ref{tLand}, this entails the claimed convergence in $L^2(\Omega,W^{1,\infty}(Q))$
\begin{equation}\label{e.quant-conv}
\expec{\sup_Q \Big(|u_\eta-u|^2+|\nabla (u_\eta-u)|^2\Big)}^\frac12 \,\lesssim\, \eta.
\end{equation}

\subsection{Control of covariances:  proof of Theorem~\ref{tLandCorr}}

Since $\omega=\{\omega_z\}_{z\in \Z^d}$ is iid, we have the covariance inequality
for any measurable maps $X$ and $Y$ of $\omega$
\begin{equation}\label{e.cov}
|\cov{X}{Y}| \,\le \, \sum_{z\in \Z^d} \expec{(X-X_z)^2}^\frac12 \expec{(Y-Y_z)^2}^\frac12,
\end{equation}
where $X_z:=X(\omega^z)$ with $\omega^z_{z'}=\omega_{z'}$ for all $z' \ne z$ and $\omega^z_z$
distributed according to $F$ independently of $\omega$.
This is the main  ingredient in the proof Theorem~\ref{tLandCorr}.

\medskip

We start with the covariance of the function $u$ itself.
By~\eqref{e.quant-conv}, it is enough to prove the claim for $u_\eta$ and pass to the limit $\eta\downarrow 0$ in the estimates.
We shall apply \eqref{e.cov} to $X= u_\eta(x)$ and $Y= u_\eta(y )$ for $x,y\in \R^d$.
Denote by $u_\eta^z$ the modifed landscape function with $\omega$ replaced by $\omega^z$ (which exists on a deterministic basis). The function $\delta_z u_\eta := u_\eta-u_\eta^z$ solves
$$
-\triangle \delta_z u_\eta + (\lambda V +\eta)\delta_z u_\eta = \lambda (\omega_z-\omega) \varphi(\cdot-z) u^z_\eta. 
$$
By the Green representation formula,
we obtain (recall that $\varphi$ is supported in $Q$)
\begin{eqnarray*}
{\expec{   |\delta_z u_\eta(x)| ^2}^\frac12}
& =&  \lambda \expec{ \Big(  \int  G_\eta(x,y)(\omega_z-\omega)(y) \varphi(y-z) u^z(y)dy\Big)^2}^\frac12
\\
&=& \lambda \expec{ (\omega_z-\omega)^2 \Big( \int_{ Q} G_\eta(x,y'+z) \varphi(y') u^z_\eta(y'+z)dy' \Big)^2}^\frac12.
\end{eqnarray*}
Since $\varphi$ is bounded and $|\omega_z-\omega|\le 1$, this yields by Theorem~\ref{tGrAnt} and~\eqref{e.land-bound-mu} (recall that $u^z$ has the same law as $u$),
\begin{eqnarray*}
{\expec{  |\delta_z u_\eta(x)|^2}^\frac12 }
&\lesssim& \expec{ \Big(\int_{ Q} G_\eta(x,y'+z)dy'\Big)^4}^\frac14
\expec{\Big(\sup_Q u^z_\eta(z+\cdot)\Big)^4}^\frac14 \\
&\lesssim&  \exp(-c (\sqrt{\lambda}\wedge 1)   |z-x|).
\end{eqnarray*}
Inserting this bound into \eqref{e.cov} applied to $X= u_\eta(x) $ and $Y=  u_\eta(y )$,
we obtain
\begin{multline}\label{e.cov-u}
|\cov{u_\eta(x)}{ u_\eta(y ) }| \,\lesssim\, \sum_{z\in \Z^d} \exp(- c(\sqrt{\lambda}\wedge 1)  |x-z|)\exp(-c (\sqrt{\lambda}\wedge 1)  |z-y|) \\
\lesssim\,   \exp(-c (\sqrt{\lambda}\wedge 1) |x-y|),
\end{multline}
which yields Theorem~\ref{tLandCorr} for $u$ by passing to the limit $\eta \downarrow 0$.

\medskip

We now turn to  $\frac1u$.
Since for all $\eta \le \lambda$, $u \ge u_\eta \ge \int  \tilde G_\lambda  = :\tau_\lambda>0$ (a deterministic positive lower bound) by the maximum principle,
\eqref{e.quant-conv} also entails 
\begin{equation}\label{e.quant-conv-inv}
\expec{\sup_Q  \Big|\frac1{u_\eta}-\frac 1u\Big|^2}\lesssim \eta,
\end{equation}
so that it is enough to bound the covariance of $\frac1{u_\eta}$ and pass to the limit $\eta \downarrow 0$. We will consider a slightly more general case. 

Let $g$ be a twice-differentiable function.
By the Taylor formula
$$
\delta_z g(u_\eta):= g(u^z_\eta)-g(u_\eta) = g(u_\eta-\delta_z u_\eta)-g(u_\eta) = -g'(u_\eta) \delta_z u_\eta + \frac12g''(\alpha_z) (\delta_z u_\eta)^2,
$$
with $\alpha_z := t u_\eta+(1-t) u^z_\eta$ for some $t\in [0,1]$.
Applied to $g(s)=1/s^2$, $s>0$, this yields
$$
\delta_z g(u_\eta)=\frac1{u^2_\eta} \delta_z u_\eta + \frac1{\alpha_z^3} (\delta_z u_\eta)^2,
$$
since $\frac1{u_\eta}< \frac1{\tau_\lambda}<\infty$ and the same is true for $u^z_\eta$.
Hence, as before, \eqref{e.cov-u} holds for $\frac1{u_\eta}$ in the form 
\begin{equation}\label{e.cov-1/u}
\left|\cov{\frac1{u_\eta(x)}}{\frac1{u_\eta(y )} }\right| \,\lesssim\,  \exp(- c(\sqrt{\lambda}\wedge 1) |x-y|),
\end{equation}
which yields Theorem~\ref{tLandCorr} for $\frac1u$ by passing to the limit $\eta \downarrow 0$.

\medskip

Since $\nabla \log u = \frac{\nabla u}u$, by   \eqref{e.quant-conv-inv},
it is enough to prove bounds on the covariance of $\nabla \log u_\eta$ and pass to the limit $\eta \downarrow 0$.
By the ``discrete'' Leibniz rule,
\begin{equation}\label{e.vert-der-drift}
\delta_z \nabla \log u_\eta \,=\, (\nabla \delta_z u_\eta) \frac1{u_\eta} -(\delta_z u_\eta)\frac{\nabla u_\eta^z}{u_\eta u_\eta^z} .  
\end{equation}
By \eqref{e.unif-grad} and Theorem~\ref{tLand} and since $\frac1{u_\eta}< \frac1{\tau_\lambda}<\infty$, we have for all $p\ge 1$,
$$
\expec{\sup_Q\Big|\frac{\nabla u_\eta^z}{u_\eta u_\eta^z}\Big|^p}^\frac1p \lesssim 1,
$$
which allows us to treat the second right-hand side of \eqref{e.vert-der-drift} as we did before.
For the first right-hand side term, we distinguish two regimes.
If $|x-z|\le 2$, we use the triangle inequality and obtain
$$
\expec{|\delta_z \nabla u_\eta(x)|^2}^\frac12 \le 2 \expec{|\nabla u_\eta(x)|^2}^\frac12 \, \lesssim \, 1.
$$
If $|x-z|>2$, we differentiate the Green representation formula
and obtain as for $u_\eta$ itself 
\begin{eqnarray*}
{\expec{ |\delta_z \nabla u_\eta(x)|^2}^\frac12 }
&\lesssim& \expec{ \Big(\int_{Q} |\nabla_x G_\eta(x ,y'+z)|dy'\Big)^4}^\frac14
\expec{\Big(\sup_Q u_\eta^z(z+\cdot)\Big)^4}^\frac14 .
\end{eqnarray*}
Since $G_\eta(x,\cdot)$ satisfies $-\triangle G_\eta(x,y'+z)+(\lambda V +\eta)G_\eta(x,y'+z)=0$
for $y'+z$ away from $x$, we may differentiate the equation with respect to $x$ and use elliptic regularity to the effect that
$$
\int_{Q} |\nabla_x G_\eta(x,y'+z)|dy' \,\lesssim\, \int_{(-2,2)^d} G_\eta(x,y'+z) dy'dx.
$$
At this point, we may conclude as for $u_\eta$ itself, and finally obtain the desired
variant of \eqref{e.cov-u} for $\nabla \log u_\eta$  
\begin{equation}\label{e.cov-nablalogu}
|\cov{\nabla \log u_\eta(x)}{\nabla \log u_\eta(y) }| \,\lesssim\,   \exp(-c (\sqrt{\lambda}\wedge 1) |x-y|),
\end{equation}
which yields Theorem~\ref{tLandCorr} for $\nabla \log u$ by passing to the limit $\eta \downarrow 0$.

\section*{Acknowledgements}
The authors warmly thank Ron Peled and Felipe Hernandez for pointing out results on first passage percolation,
and Christophe Garban for the early references on such results.
AG acknowledges financial support from the European Research Council (ERC) under the European Union's Horizon 2020 research and innovation programme (Grant Agreement n$^\circ$~864066). 
 GD is partially supported by the European Community H2020 grant GHAIA 777822,
and the Simons Foundation grant 601941, GD.
SM is supported in part by the NSF grant DMS 1839077 and Simons Foundation grant 563916, SM.

\bibliographystyle{plain}
%\bibliography{biblio}

\Addresses

\end{document}